\documentclass[review]{elsarticle}
\showoutput
\usepackage{lscape}
\usepackage[title,titletoc,toc]{appendix}
\usepackage[colorlinks=true,linkcolor=black, citecolor=blue, urlcolor=blue]{hyperref}
\usepackage{amssymb}
\usepackage{amsmath}
\usepackage{lineno,hyperref}
\usepackage{setspace}
\usepackage{graphicx}
\usepackage{longtable}
\usepackage{fixltx2e}
\newtheorem{thm}{Theorem}[section]

\newtheorem{cor}[thm]{Corollary}
\newtheorem{lemma}[thm]{Lemma}
\newdefinition{defi}{Definition}
\newproof{proof}{Proof}
\newproof{pot}{Proof of Theorem \ref{thm2}}
\newdefinition{rmk}{Remark}
\newtheorem{conjecture}{Conjecture}

\journal{Finite Fields and Their Applications}

\begin{document}

\begin{frontmatter}

\title{\textbf{ Primitive Element Pairs with One Prescribed Trace over a Finite Field}}
\author[iit]{Anju Gupta\corref{cor1}}
\ead{anjugju@gmail.com}
\author[iit]{R. K. Sharma}
\ead{rksharmaiitd@gmail.com}
\author[sdc]{Stephen  D. Cohen \corref{fn1}}
\ead{Stephen.Cohen@glasgow.ac.uk}
\cortext[cor1]{Corresponding author}
\cortext[fn1]{Formerly, Professor of Number Theory, University of Glasgow}
%\cortext[cor2]{Principal corresponding author}
%\fntext[fn1]{This is the specimen author footnote.}
%\fntext[fn2]{Another author footnote, but a little more
%longer.}
%\fntext[fn3]{Yet another author footnote. Indeed, you can have
%any number of author footnotes.}
\address[iit]{Department of Mathematics, Indian Institute of Technology Delhi,\\
		New Delhi, 110016, India}
\address[sdc]{6 Bracken Road, Portlethen, Aberdeen AB12 4TA, Scotland}
%\address[focal]{River Valley Technologies, 9, Browns Court,
%Kennford, Exeter, United Kingdom}
%\address[els]{Central Application Management,
%Elsevier, Radarweg 29, 1043 NX\\
%Amsterdam, Netherlands}
%\author{Anju\fnref{fn1}}
%\ead{anjugju@gmail.com}
%\address{Department of Mathematics, Indian Institute of Technology Delhi,\\
%		New Delhi, 110016, India}

%\author[mymainaddress]{R. K. Sharma}
%\ead[url]{rksharmaiitd@gmail.com}
%\author{R. K. Sharma\fnref{fn2}}
%\ead{rksharmaiitd@gmail.com}
%\address{Department of Mathematics, Indian Institute of Technology Delhi,\\
%		New Delhi, 110016, Indi}a
%\author[mysecondaryaddress]{Globa lCustomer Service\corref{mycorrespondingauthor}}


\begin{abstract}
In this article, we establish a sufficient condition for the existence of a primitive element $\alpha \in {\mathbb{F}_{q^n}}$ such that the element $\alpha+\alpha^{-1}$ is also a  primitive element of ${\mathbb{F}_{q^n}},$ and $Tr_{\mathbb{F}_{q^n}|\mathbb{F}_{q}}(\alpha)=a$ for any prescribed $a \in \mathbb{F}_q$, where $q=p^k$ for some prime $p$ and positive integer $k$. We prove that every finite field $\mathbb{F}_{q^n}~ (n \geq5),$ contains such primitive elements except for finitely many values of $q$ and $n$. Indeed, by computation, we conclude that there are no actual exceptional pairs $(q,n)$ for $n\geq5.$
%$\alpha$ such that the element $\alpha+\alpha^{-1}$ is a  primitive element of ${\mathbb{F}_{q^n}},$ and $Tr_{\mathbb{F}_{q^n}|\mathbb{F}_{q}}(\alpha)=a$ for any prescribed $a \in \mathbb{F}_q^*$.
\end{abstract}

\begin{keyword}
\texttt{Finite Field\sep Character\sep Primitive Element}
\MSC[2010] 12E20 \sep 11T23
\end{keyword}

\end{frontmatter}

%\linenumbers

\section{Introduction}

 Let $\mathbb{F}_q$ denote the finite field of order $q=p^k$ %Let $\mathbb{F}_q$ be a finite field of order $q=p^k,$
for some prime $p$ and some positive integer $k,$ and $\mathbb{F}_{q^n}$ denotes an extension of $\mathbb{F}_q$ of degree $n.$ The multiplicative group $\mathbb{F}_q^*$ of $\mathbb{F}_q$ is cyclic and its generators   are called \textit{primitive elements} of $\mathbb{F}_q$.  The field ${\mathbb{F}_q}$ has $\phi(q-1)$ primitive elements, where $\phi$ is the Euler's phi-function.

For $\alpha\in \mathbb{F}_{q^n}$, the \textit{trace} $Tr_{\mathbb{F}_{q^n}|\mathbb{F}_{q}}(\alpha)$ of $\alpha$ is defined by $Tr_{\mathbb{F}_{q^n}|\mathbb{F}_{q}}(\alpha)=\alpha+\alpha^q+\ldots+\alpha^{q^{n-1}}$.

In general, for any primitive element $\alpha \in \mathbb{F}_q,$ $f(\alpha)$ (where $f$ is any rational function) need not be primitive in $\mathbb{F}_q,$ for example, if we take the polynomial function $f(x)=x+1$ over the field $\mathbb{F}_2$ of order 2 then $1$ is the only primitive element of $\mathbb{F}_2,$ but $f(1)=0,$ which is not primitive. But for $f(x)=\frac{1}{x},$ $f(\alpha)$ is primitive in $\mathbb{F}_q$ whenever $\alpha$ is primitive. We call $(\alpha,f(\alpha))$ a \textit{primitive pair} if both $\alpha$ and $f(\alpha)$ are primitive.  Much work has been done in this direction. In 1985, Cohen \cite{s.d} proved the existence of  two consecutive primitive elements in $\mathbb{F}_q$ with $q>3, $ $q \not\equiv 7 \mod \ {12},$ and $q \not\equiv 1
\mod \ {60}$.
Chou and Cohen \cite{chou} completely resolved the question whether there exists a primitive element $\alpha$ such that $\alpha$ and $\alpha^{-1}$ both have trace zero over ${\mathbb{F}_q}.$
He and Han \cite{he} studied primitive elements of the form  $\alpha+\alpha^{-1}$ over finite fields.  In 2012, Wang et al. \cite{wang} established  a sufficient condition for the existence of $\alpha$ such that $\alpha $ and $\alpha +\alpha^{-1}$ are both primitive, and also a sufficient condition for the existence of a primitive normal element $\alpha$ such that $\alpha +\alpha^{-1}$ is primitive for the case $2|q.$ Liao et al. \cite{liao} generalized their results to the case when $q$ is any prime power. In 2014, Cohen \cite{cohens} completed the existence results obtained by Wang et al. \cite{wang} for finite fields of characteristic 2. In \cite{cohensd},
Cohen  proved that for every $a \in \mathbb{F}_q,$  $\mathbb{F}_{q^n}$ contains a primitive element $\alpha$   such that $Tr_{\mathbb{F}_{q^n}|\mathbb{F}_{q}}(\alpha) = a,$ if $n \geq 3,$ and $(q, n) \neq (4, 3).$    Moreover, if $n = 2$ or $(q, n) =
(4, 3),$ for every nonzero $a \in \mathbb{F}_{q}^*,$ there exists a primitive element $\alpha\in \mathbb{F}_{q^n}$ such that
$Tr_{\mathbb{F}_{q^n}|\mathbb{F}_{q}}(\alpha) = a.$ In 2014, Cao and Wang \cite{cao} proved that for all $q$ and $n \geq 29$, $\mathbb{F}_{q^n}$ contains an element $\alpha$ such that $\alpha+\alpha^{-1}$ ia also  primitive, and $Tr_{\mathbb{F}_{q^n}|\mathbb{F}_q}(\alpha)=a,$ $Tr_{\mathbb{F}_{q^n}|\mathbb{F}_q}(\alpha^{-1})=b$
for any pair of prescribed $a, b\in \mathbb{F}_q^*$.

 In this article, we consider the existence of a primitive pair  $(\alpha,\alpha+\alpha^{-1})$ in $\mathbb{F}_{q^n}$ with $Tr_{\mathbb{F}_{q^n}|\mathbb{F}_{q}}(\alpha)=a$ for any prescribed $a \in \mathbb{F}_q.$
 %Clearly, $(q,1)\not\in\mathfrak{P}$ as in that case $Tr_{\mathbb{F}_{q^n}|\mathbb{F}_{q}}(\alpha)=\alpha$. Hence for $(q,1)$ to be in $\mathfrak{P}$ every pair $(\alpha,\alpha+\alpha^{-1})$ in $\mathbb{F}_q$ must be primitive, which is possible only if $q-1$ is prime. Moreover if $q-1$ is prime then $p=2$. Hence $(1,0)$ must be a primitive pair, which is not possible. Also if $n=2,$ then there is no primitive element with trace 0. Hence  $(q,2)\not\in\mathfrak{P}$. Thus we may assume that $n \geq 3.$
 Precisely, we prove the following main result.
 \begin{thm}\label{mr}
 Suppose $q=p^k$ for some positive integer  $k$ and a prime number $p$. Also suppose $n \geq 5$ is a natural number. Then $\mathbb{F}_{q^n}$ contains a primitive pair  $(\alpha,\alpha+\alpha^{-1})$ in $\mathbb{F}_{q^n}$ with $Tr_{\mathbb{F}_{q^n}|\mathbb{F}_{q}}(\alpha)=a$ for any prescribed $a \in \mathbb{F}_q$ unless one of the following holds:
 \begin{enumerate}
 \item $n=5$, and $2 <q \leq 16$ or $q=19,25,31,37,43,49,61,71;$
 \item $n=6$, and $2 \leq q \leq 25$ or $q=29,31,61;$
 \item $n=7$, and  $q=3,4,7;$
 \item $n=8$, and $q=2,3,4,5,8;$
 \item $n=9,12$, and $q=2,3;$
 \item $n=10$, and $q=2.$
 \end{enumerate}
 \end{thm}
 From Theorem \ref{mr}, through computation, we have established the following corollary.
 \begin{cor}\label{cor1}
Let $q=p^k$ for some positive integer $k ,$ and  prime $p$. Also suppose that $n\geq5$ is a positive integer. Then for every $a\in \mathbb{F}_q,$ $\mathbb{F}_{q^n}$ contains a primitive element $\alpha$ such that $\alpha+\alpha^{-1}$ is also primitive and $Tr_{\mathbb{F}_{q^n}|\mathbb{F}_q}(\alpha)=a$. %except when $(q,n)=(4,3).$
\end{cor}
 Throughout rest of the paper, we shall use the notation  $\mathfrak{P}$ for  the set of $(q,n)$ ($q=p^k$ for any positive integer $k$) such that  $\mathbb{F}_{q^n}$ contains a primitive pair $(\alpha,\alpha+\alpha^{-1})$, with $Tr_{\mathbb{F}_{q^n}|\mathbb{F}_{q}}(\alpha)=a$ for any prescribed $a \in \mathbb{F}_q.$

Clearly, $(q,1)\not\in\mathfrak{P}$ as in that case $Tr_{\mathbb{F}_{q^n}|\mathbb{F}_{q}}(\alpha)=\alpha$. Hence for $(q,1)$ to be in $\mathfrak{P},$ every pair $(\alpha,\alpha+\alpha^{-1})$ in $\mathbb{F}_q$ must be primitive, which is possible only if $q-1$ is prime. Moreover if $q-1$ is prime then $p=2$. Hence $(1,0)$ must be a primitive pair, which is not possible. Also if $n=2,$ then there is no primitive element with trace 0. Hence  $(q,2)\not\in\mathfrak{P}$. Thus we may assume that $n \geq 3.$ For the sake of simplicity, we have not dealt with the cases $n=3$ and $4$ in this article, although we intend to return to them in a future paper.

\section{Preliminaries}

In this section, we give some necessary definitions,  and results which will be used throughout. For basics on finite fields, and additive and multiplicative characters of finite fields, reader is referred to \cite{nieder}.
Throughout the section, $q$ is an arbitrary prime power. For any positive integer $m>1$, we use the notation $\omega(m)$ for the number of prime divisors of $m$. Also $W(m)$  denotes the number of square free divisors of $m$, i.e., $W(m)=2^{\omega(m)}.$

\begin{defi}
Let $e|q - 1$. An element  $\xi \in {\mathbb{F}_{q}^{*}}$  is called  \textit{$e$-free} if $\xi = \gamma^d$ for any $d|e$, and $\gamma \in \mathbb{F}_q$  implies $d = 1.$ %i.e. if $\gcd(d,\frac{q-1}{ord_q(\alpha)})=1.$
Hence an element $\alpha \in {\mathbb{F}_{q}^{*}}$ is primitive if and only if it is $(q-1)$-free.
\end{defi}

Following Cohen and Huczynska \cite{cohen.h, cohen.s}, it can be shown that for any $m|q-1,$
$$\rho_m:
\alpha \mapsto \theta(m)\sum\limits_{d|m}\frac{\mu(d)}{\phi(d)}\sum\limits_{\chi_d}\chi_d(\alpha),$$
where $\theta(m) := \frac{\phi(m)}{m},$ $\mu$ is M\"obius
 function and the internal sum runs over all multiplicative characters $\chi_d$ of order $d,$ gives  an expression of the characteristic function for the subset of $m$-free elements of ${\mathbb{F}_{q}^{*}}$.

 An expression of the characteristic function for the set of  elements in $\mathbb{F}_{q^n}$ with $Tr_{\mathbb{F}_{q^n}|\mathbb{F}_{q}}(\alpha)=a\in \mathbb{F}_q$ is given by,
$$
\tau_a:\alpha \mapsto \frac{1}{q}\sum\limits_{\psi\in \widehat{\mathbb{F}_q}}\psi{(Tr_{\mathbb{F}_{q^n}|\mathbb{F}_{q}}(\alpha)-a)},$$
where the sums are over all additive characters $\psi$ of $\mathbb{F}_{q}$ , i.e., all members of $\widehat{F_{q}}$.

{Since every additive character $\psi$ of $\mathbb{F}_q$ can be obtained by $\psi(\alpha)=\psi_0(u\alpha)$, where $\psi_0$ is the canonical additive character of $\mathbb{F}_q$ and $u$ is any element of $\mathbb{F}_q$, then}
\begin{align}
\tau_a(\alpha)&=\frac{1}{q}\sum\limits_{u\in {\mathbb{F}_q}}\psi_0{(Tr_{\mathbb{F}_{q^n}|\mathbb{F}_{q}}(u\alpha)-ua)}\nonumber\\
&=\frac{1}{q}\sum\limits_{u\in {\mathbb{F}_q}}\hat{\psi_o}{(u\alpha)\psi_0(-ua)},
\end{align}
where $\hat{\psi_0}$ is the additive character of $\mathbb{F}_{q^n}$ defined by
 $\hat{\psi_0}(\alpha) =\psi_0( Tr_{\mathbb{F}_{q^n/\mathbb{F}_q}}(\alpha))$.

Next, we give some lemmas, which will be used in our main results.
\begin{lemma}\cite[\em{Theorem 5.4}]{nieder}\label{a}
If $\chi$ is any non-trivial character of a finite abelian group $G,$ and $\beta$ is a non-trivial element of $G$ then
$$\sum\limits_{\beta\in G}\chi(\beta)=0~~\text{and}~~ \sum\limits_{\chi\in \widehat{G}}\chi(\beta)=0.$$
\end{lemma}

\begin{lemma}\cite{castro}\label{m}
Let $\chi$ be a non-trivial multiplicative character of order $r$ and $\psi$ be a non-trivial additive character of $\mathbb{F}_{q^n}$. Let $f,~g$ be rational functions in $\mathbb{F}_{q^n}(x)$ such that $f\neq yh^r,$ for any $y\in \mathbb{F}_{q^n}$,  $h\in \mathbb{F}_{q^n}(x)$, and $g\neq h^p-h+y$ for any $y\in \mathbb{F}_{q^n}$,  $h\in \mathbb{F}_{q^n}(x).$ Then
$$\big|\sum_{x\in \mathbb{F}_{q^n}\backslash S}
\chi(f(x))\psi(g(x))\big| \leq (\deg(g)_{\infty}+m+m'-m''-2)q^{n/2},$$
where $S$ is the set of poles of $f$ and $g,$ $(g)_{\infty}$ is the pole divisor of $g,$ $m$ is the number of distinct zeros and finite poles of $f$ in $\bar{\mathbb{F}}_q$ (algebraic closure of $\mathbb{F}_q$), $m'$ is the number of distinct poles of $g$ (including $\infty$) and $m''$ is the number of finite poles of $f$ that are poles or zeros of $g.$
\end{lemma}

\section{Existence of Primitive Pairs $(\alpha,\alpha+\alpha^{-1})$ in $\mathbb{F}_{q^n}$ with $Tr(\alpha)=a$}
In this section, for every $a \in \mathbb{F}_{q}$, we find a sufficient condition for the existence of primitive pairs $(\alpha,\alpha+\alpha^{-1})$ in $\mathbb{F}_{q^n}$ such that $Tr_{\mathbb{F}_{q^n}|\mathbb{F}_{q}}(\alpha)=a$.\\
%We begin by proving a series of results.\\
Let $l_1,~l_2|q^n-1$. For any $a \in \mathbb{F}_q$, let $N_a(l_1,l_2)$
be the number of $\alpha \in \mathbb{F}_{q^n}$ such that $\alpha$ is $l_1$-free, $\alpha+\alpha^{-1}$ is $l_2$-free and $Tr_{\mathbb{F}_{q^n}|\mathbb{F}_{q}}(\alpha)=a$.   Hence we need to show that $N_a(q^n-1,q^n-1)>0$ for every $a\in\mathbb{F}_q$.
		\begin{lemma}\label{p}
			Let $a \in \mathbb{F}_q,$  and $l_1,~l_2|q^n-1$. Then $N_a(l_1,l_2)>0$ if $q^{n/2-1} >C_qW(l_1)W(l_2)$,
where \[ C_q :=  \left\{
\begin{array}{ll}
      3 & \text{if}~ q~\text{is odd} \\
      2 & \text{if}~q~\text{is even}.\\
\end{array}
\right. \]
\end{lemma}
\begin{proof} By definition,
\begin{equation}\label{1}
N_{a}(l_1,l_2)=\sum_{\alpha\in \mathbb{F}_{q^n}^*}\rho_{l_1}(\alpha)\rho_{l_2}(\alpha+\alpha^{-1})\tau_a(\alpha).
\end{equation}

Now \eqref{1} gives
\begin{equation}\label{2}	 N_a(l_1,l_2)=\frac{\theta(l_1)\theta(l_2)}{q}\sum_{d_1|l_1,~d_2|l_2}\frac{\mu(d_1)}{\phi(d_1)}\frac{\mu(d_2)}{\phi(d_2)}\sum_{\chi_{d_1},\chi_{d_2}}\pmb{\chi}_
a(\chi_{d_1},\chi_{d_2}),
\end{equation}
where
\begin{align*}
\pmb{\chi}_a(\chi_{d_1},\chi_{d_2})&=\sum_{u\in{{\mathbb{F}_q}}}\psi_0(-au)\sum_{\alpha \in\mathbb{F}_{q^n}^*}\chi_{d_1}(\alpha)\chi_{d_2}(\alpha+\alpha^{-1})\hat\psi_0(u\alpha).\\
\end{align*}
As we know that  $\chi_{d_i}(x)=\chi_{q^n-1}(x^{n_i})$ for $i=1,~2,$ and some  $n_i \in \{0, 1,2,\cdots,q^n-2\}.$  Thus
\begin{align*}
\pmb{\chi}_a(\chi_{d_1},\chi_{d_2})&=\sum_{u\in{{\mathbb{F}_q}}}\psi_0(-au)\sum_{\alpha\in\mathbb{F}_{q^n}^*}\chi_{q^n-1}(\alpha^{n_1}(\alpha^2+1)^{n_2}(\alpha)^{q^n-n_2-1})\hat\psi_0(u\alpha)\\
		&=\sum_{u\in{{\mathbb{F}_q}}}\psi_0(-au)\sum_{\alpha \in\mathbb{F}_{q^n}^*}\chi_{q^n-1}(F(\alpha))\hat{\psi}_0(u\alpha),
\end{align*}
where $F(x)=x^{n_1+q^n-1-n_2}(x^2+1)^{n_2}\in \mathbb{F}_{q^n}[x]$ for some $0 \leq n_1,n_2 < q^n-1.$

If $F(x) \neq yH^{q^n-1}$ for any $y \in \mathbb{F}_{q^n}$ and $H \in \mathbb{F}_{q^n}[x]$ then using Lemma \ref{m}, if $q$ is odd,
$$|\pmb{\chi}_a| \leq (4-1)q^{n/2}=3q^{n/2}.$$
On the other hand, if $q$ is even, then $x^2 +1 = (x+1)^2$ and this can be sharpened to
$$|\pmb{\chi}_a| \leq (3-1)q^{n/2}=2q^{n/2},$$
%For this reason,  henceforth we define $C_q$ to be $3$ if $q$ is odd and $2$ if $q$ is even.
i.e., $$|\pmb{\chi}_a| \leq C_qq^{n/2}.$$
So let $F=yH^{q^n-1}$ for some $y \in \mathbb{F}_{q^n}$ and $H \in \mathbb{F}_{q^n}[x]$. Then
\begin{equation}\label{3}
  x^{n_1+q^n-1-n_2}(x^2+1)^{n_2}=yH(x)^{q^n-1},
\end{equation}
for some $y \in \mathbb{F}_{q^n}$ and $H \in \mathbb{F}_{q^n}[x].$ Now
\eqref{3} implies that $ (x^2+1)^{n_2}|H^{q^n-1}.$  Hence $n_2=0$ or

\begin{equation}\label{4}
			 x^{n_1+q^n-1-n_2}=y(x^2+1)^{q^n-1-n_2}A^{(q^n-1)},
\end{equation}
where $A(x)=H(x)/(x^2+1) \in \mathbb{F}_{q^n}[x].$ From
\eqref{4}, we observe  that  $(x^2+1)^{q^n-1-n_2}|x^{n_1+q^n-1-n_2}$, which is possible only if $q^n-1=n_2,$ a contradiction. Hence $n_2=0.$ Putting this in \eqref{3}, we get
$ x^{n_1+q^n-1}=yH^{q^n-1}\Rightarrow n_1=(k_1-1)(q^n-1),$ where $k_1$ is the degree of $H(x).$ This is possible only if $k_1=1$ and hence $n_1=0.$
Thus, in this case $(\chi_{d_1},\chi_{d_2})=(\chi_{1},\chi_{1}).$ Additionally if, $u \neq 0$ then using Lemma \ref{a}, we get $$ |\pmb{\chi}_a(\chi_{d_1},\chi_{d_2})|=q-1 \leq C_q q^{n/2+1}.$$

Hence
$|\pmb{\chi}_a(\chi_{d_1},\chi_{d_2})|\leq C_q q^{n/2+1},$ when $(\chi_{d_1},\chi_{d_2},u)\neq (\chi_1,\chi_1,0).$
Thus, using \eqref{2} we get
\begin{equation}\label{5}
N_a(l_1,l_2)\geq \frac{\theta(l_1)\theta(l_2)}{q}(q^n-1-C_q q^{n/2+1}(W(l_1)W(l_2)-1)).
\end{equation}
Hence $N_a(l_1,l_2)> 0 $ if $q^{n/2}>q^{-n/2+1}+C_q q(W(l_1)W(l_2)-1),$ i.e., if $q^{n/2-1} >C_qW(l_1)W(l_2).$
Hence the result follows.
\end{proof}
In the next lemma, we give upper bounds for the absolute values of $N_a(sl,l)-N_a(l,l)$ and $N_a(l,sl)-N_a(l,l)$.
\begin{lemma}\label{z}
Let $l|q^n-1$ and $s$ any prime dividing $q^n-1$ but not $l$. Then
$$|N_a(sl,l)-N_a(l,l)|\leq \frac{C_q\theta(l)^2\theta(s)}{q}W(l)^2q^{n/2+1}.$$ Also
$$|N_a(l,sl)-N_a(l,l)|\leq \frac{C_q\theta(l)^2\theta(s)}{q}W(l)^2q^{n/2+1}.$$
\end{lemma}
\begin{proof}
By definition, we have

$N_a(sl,l)-N_a(l,l)=\frac{\theta(l)^2\theta{(s)}}{q}\Big\{\sum_{s|d_1|sl}\sum_{d_2|l}\frac{\mu(d_1)}{\phi(d_1)}\frac{\mu(d_2)}{\phi(d_2)}\sum_{\chi_{d_1},\chi_{d_2}}\pmb{\chi}_
a(\chi_{d_1},\chi_{d_2})\Big\}.\\$

Using $|\pmb{\chi}_
a(\chi_{d_1},\chi_{d_2})|\leq C_qq^{n/2+1},$ we get
$$|N_a(sl,l)-N_a(l,l)|\leq \frac{\theta(l)^2\theta{(s)}}{q}C_q q^{n/2+1}W(l)\{W(sl)-W(l)\}.$$
Since $W(sl)=2W(l),$ we get
$$|N_a(sl,l)-N_a(l,l)|\leq \frac{\theta(l)^2\theta{(s)}}{q}C_qq^{n/2+1}W(l)^2.$$
Similarly
$$|N_a(l,sl)-N_a(l,l)|\leq \frac{C_q\theta(l)^2\theta(s)}{q}W(l)^2q^{n/2+1}.$$
\end{proof}
Next, we obtain an extension of the sieving Lemma 3.7 of \cite{cohens}. The  proof follows on the lines  of  Proposition 5.2 of \cite{gkape}, but  is given again for completeness..
\begin{lemma}\label{bb}
Suppose $l|q^n-1$ and  $\{p_1,\ldots,p_r\}$ is the collection of all the primes dividing $q^n-1$ but not $l $. Then
\begin{equation}\label{*eq}
N_a(q^n - 1, q^n - 1) \geq \sum_{i=1}^{r}N_a(p_il, l) +\sum_{i=1}^{r}N_a(l, p_il) - (2r - 1)N_a(l, l).
\end{equation}
\end{lemma}
\begin{proof}
   %It will become number (7) (so that all subsequent numbers will increase by 1).  Proceed to the proof something as follows.

The left side of \eqref{*eq} counts every  $\alpha \in \mathbb{F}_{q^n}$ for which $\alpha$ has trace $a$ and both $\alpha$ and $\alpha +1/\alpha$ are primitive. Thus, it counts $1$ for every $\alpha$ for which $\alpha$ has trace $a$, both $\alpha$ and $\alpha+ 1/\alpha$ are $l$-free, and for each $i=1, \ldots, r$, both $\alpha$ and $\alpha+ 1/\alpha$ are $p_i$-free.  Observe that the right side of \eqref{*eq} scores $1$ for each such $\alpha$, whereas, for any other $\alpha \in \mathbb{F}_{q^n}$  it scores an integer $\leq 0$.  This completes the proof.

\end{proof}

By taking $l_1=l_2=q^n-1$ in Lemma \ref{p}, we see that $(q,n)\in \mathfrak{P},$ if
\begin{equation}\label{***}
q^{n/2-1}> C_qW(q^n-1)^2.
\end{equation}
 We further improve this criterion.
\begin{thm}\label{aa}
Let $l|q-1$ and $\{p_1,~p_2, \ldots,p_r\}$ be the collection of all  the primes dividing $q^n-1,$ but not $l$. Suppose $\delta =1-2\sum_{i=1}^{r}\frac{1}{p_i}$ and $\Delta=\frac{2r-1}{\delta}+2$ and assume $\delta >0.$ If
\begin{equation}\label{****}
q^{n/2-1} > C_qW(l)^2\Delta,
\end{equation}
then $(q,n)\in \mathfrak{P}$.
\end{thm}
\begin{proof}
From Lemma \ref{bb}, we deduce that
\begin{align}
N_a(q^n-1,q^n-1)\geq & \sum_{i=1}^{r}\{N_a(p_il,l)-\theta(p_i)N_a(l,l)\}+\{N_a(l,p_il)-\theta(p_i)N_a(l,l)\}\nonumber\\
&+\delta N_a(l,l).
\end{align}
Using Lemma \ref{z}, we get
\begin{align*}
N_a(q^n-1,q^n-1)\geq &\frac{\theta(l)^2}{q}\Big\{\sum_{i=1}^{r}2\theta{(p_i)}(-C_qq^{n/2+1}W(l)^2)\\
&+\delta\{q^n-1-C_qq^{n/2+1}(W(l)^2-1)\}\Big\}.
\end{align*}

\begin{align*}
N_a(q^n-1,q^n-1)\geq& \frac{\theta(l)^2}{q}\delta\Bigg\{\Big(\frac{2\sum_{i=1}^{r}\theta(p_i)}{\delta}+1\Big)\{-C_qq^{n/2+1}W(l)^2\}\\
&+\{q^n-1+C_q q^{n/2+1}\}\Bigg\}.
\end{align*}
Using $\delta=2\sum_{i=1}^{s}\theta(p_i)-(2s-1)$, we get
\begin{align*}
N_a(q^n-1,q^n-1)&\geq \frac{\theta(l)^2}{q}\delta\{-C_q\Delta q^{n/2+1}W(l)^2+q^n-1+C_qq^{n/2+1}\}.
\end{align*}

Since $\delta>0$, $N_a(q^n-1,q^n-1)>0$
 if $q^{n/2-1}> q^{-n/2-1}-C_q+C_qW(l)^2\Delta$, that is, if $q^{n/2-1} > C_qW(l)^2\Delta.$ So if $q^{n/2-1} > C_qW(l)^2\Delta$ then for every $a \in \mathbb{F}_q,$ $\mathbb{F}_{q^n}$ contains a primitive pair $(\alpha,\alpha+\alpha^{-1})$ such that $Tr_{\mathbb{F}_{q^n}|\mathbb{F}_q}(\alpha)=a,$ and hence $(q,n) \in \mathfrak{P}.$

\end{proof}

\section{Exploiting the condition (\ref{***}) through calculation}
From now on we abbreviate  $\omega(q^n-1)$ to $\omega$.
\begin{lemma}\label{g}
Let   $m > 2\times10^{31}$ be a positive integer. Then $W(m) < m^{2/9}.$
\end{lemma}
\begin{proof}
If $m > 2\times10^{31}$ then $m^{2/9}>2^{23}.$ Hence if $\omega(m)\leq 23$ then $W(m) < m^{2/9}.$ So assume $\omega(m) >23.$   Write $m=m_1m_2$, where $m_1$ and $m_2$ are coprime with each prime dividing $m_1$ one of the smallest $23$  primes dividing $m$ and each prime dividing $m_2$ one of the remaining $\omega(m)-23 $ primes dividing $m$. Thus $W(m)=W(m_1) W(m_2)$, where $m_1\geq 2\cdot3\ldots83$ and $W(m_2)\leq m_2^{2/9}$, since $l^{2/9} > 2$ for any prime $l \geq 89$.  Further, since $m_1 > 2\times10^{31}$,  by the above argument, $W(m_1) <m_1^{2/9}.$
 So $W(m) < m_1^{2/9}m_2^{2/9}=m^{2/9}.$
\end{proof}

\begin{thm}\label{r}

Suppose $q=p^k,$ where $k$ is a positive integer and $p$ is any prime number. Then $(q,n) \in \mathfrak{P}$ for  $n \geq 26$.
\end{thm}
\begin{proof} First suppose $q \geq 16$ (in addition to $n\geq 26$).
Then  $q^n\geq16^{26}>2\times 10^{31}$. Hence
by Lemma \ref{g}, we have
$$q^{n/2-1}/W(q^n-1)^2>q^{n/18-1}>q^{4/9}>3.4>3.$$
Hence $(q,n)\in \mathfrak{P}$ for all $q \geq 16$ and $n\geq26$.

Now suppose $q$ is a prime power with $2 \leq q \leq 13$.  Write $n_q$ for the least integer such that $q^{n_q} > 2\times 10^{31}$. Thus $n_2=104,$ $ n_3=66,$  $n_4=52 ,$ $n_5=45,$  $  n_7=38,$  $ n_8=35,$ $ n_9=33,$   $  n_{11} =31,$   $ n_{13}=29$.  Hence, as in the first part, for $n\geq n_q$,
$$q^{n/2-1}/W(q^n-1)^2 \geq q^{n_q/18-1}>3.$$

Finally, for each pair $(q,n)$ with $2 \leq q \leq 13$ and $26 \leq n < n_q$, check directly that (\ref{***}) holds by evaluating the exact value of $\omega(q^n-1) $ in each case. (The most delicate case is when $q=2, n=28, \omega(q^n-1)=6$, for this case we refer to Table \ref{t1}.)

This completes the proof.
\end{proof}

\section{Odd prime powers $q$}

Suppose $q$ is an odd prime power so that $C_q=3$. By Theorem \ref{r}, we may assume that $n \leq 25$. Initially,  we suppose $n \geq 6$.  Throughout the rest of the paper, we use $R$ to denote the value on right hand side of \eqref{****}.

To begin we give  a lemma which echoes Lemma \ref{g}.
\begin{lemma}\label{1/8} Let $m$ be a positive integer such that $\omega(m) \geq 149$. Then
$$ W(m) < m^{1/8}.$$
\end{lemma}
\begin{proof}The product of the first 149 primes (the largest being $859$) exceeds $M=7.5 \times 10 ^{358}$.  Write $m=m_1m_2$,
a product of coprime integers, where all primes dividing $m_1$  are amongst the least 149 primes dividing $m$ and those dividing $m_2$ are larger primes.   Hence $m_1>M$ and $m_1^{1/8} >M^{1/8}>7.23\times 10^{44}$, whereas $W(m_1) = 2^{149}<7.14 \times 10^{44}$.   Since $l^{1/8} > 2$ for all primes $l >859$, the result follows.
 \end{proof}
 \begin{thm}\label{n6}
Suppose $q=p^k \geq3,$ for some positive integer  $k$  and odd prime $p$. Also suppose $n \geq 6$ is a positive integer. Then $(q,n) \in \mathfrak{P}$ for  all  pairs $(q,n)$ except (possibly) the pairs $(3,7),(7,7),(3,8),(5,8),(3,9),(3,12)$ and $(q,6)$ with $3\leq q\leq25$, and $q=29,31,61$.
\end{thm}
\begin{proof}
  Assume first that $\omega \geq 149$.  By Lemma \ref{1/8} to satisfy inequality (\ref{****}), it suffices that $q^{n/2-2n/8-1}> 3$, i.e., $q^{n/4-1} >3$ which easily holds unless $q\leq 9$ if $n=6$ or  $q=3 $ if  $n=7,8$ (which would imply $\omega <149$).

We now assume that $18 \leq \omega \leq 148$.
   Then, in the situation and with the notation of Theorem \ref{aa}, take $l$ to be the factor of $q^n-1$  whose prime factors are the least $18$ primes dividing $q^n-1$. Thus $r \leq 130$.
Further, $\delta$ must be at least the value obtained when $r=130$ and $ \{p_1, \ldots, p_r\}$ comprises those primes from 67 (the $19$th prime) to 857 (the $148$th prime), inclusive.  Thus $\delta > 0.074703$  and $R< 7.1517\times 10^{14}$. Now \eqref{****} holds if  $q >R^{(2/(n-2))}$ , i.e. if $q^n >R^{(2n/(n-2))}$, so certainly if $q^n>R^3$ (since $n \geq 6$),    Hence, $q^n> 3.6579 \times 10^{44}$ suffices.  If in fact $\omega \geq 30$ then, $q^n-1$  is at least the product of the first 30 primes, in which case $q^n > 3.1600 \times 10^{46}.$   We conclude that $(q,n) \in \mathfrak{P}$ whenever  $\omega \geq 30$ or $q^n>(3.6579 \times 10^{44})^{(1/n)}$ (at worst, when $n=6,$ $q>2.6743\times 10^7$).% On the other hand,  since $n \geq 6$, $q^{n/2-1}\geq q^2$.   Hence \eqref{****} holds whenever $q>2.6742 \times 10^7$.

We next assume that $7 \leq\omega \leq 29$ (and  $q<2.6743 \times 10^7$).   Repeat the above process with $\omega(l) =7$ and $r \leq 22$.   Now,  $ \delta$ will be at least the value obtained when $r=22$ and $\{p_1, \ldots, p_r\}$ comprises those primes between 19 and 109 (inclusive).  Thus $\delta > 0.12379$ and $R<1.7171\times 10^7$.  As in the previous case, it follows that \eqref{****} holds whenever $q^n>5.0625\times 10^{21}$. Now, if also  $\omega \geq 18$ then $q^n> 1.17288\times 10^{23}.$ Hence  we conclude that $(q,n) \in \mathfrak{P}$ whenever  $\omega \geq 18$ or $q>(5.0625\times 10^{21})^{(1/n)}$; so that, at worst $(n=6)$, whenever $q>4144.$.  %(\ref{****}) holds  whenever $q>4144$  which is bound to be case if, in fact, $\omega \geq 18$.   We conclude that $(q,n) \in \mathfrak{P}$ whenever $\omega \geq18$ or $q>4144$.

Next assumed that $4 \leq \omega \leq 17$ (and $q < 4144$).

Repeat the above process with $\omega(l)=5$ and $r\leq12$. Now,  $ \delta$ will be at least the value obtained when $r=12$ and $\{p_1, \ldots, p_r\}$ comprises those primes between 13 and 59 (inclusive).  So $\delta > 0.13927$ and hence $R < 5.1348\times 10^5.$  Thus \eqref{****} holds whenever $q^n>1.35381\times 10^{17}$. Now if $\omega \geq 15$, then $q^n> 6.1148897\times 10^{17}.$ Hence  we conclude that $(q,n) \in \mathfrak{P}$ whenever  $\omega \geq 15$ or $q>(1.35381\times 10^{17})^{(1/n)}$; so that, at worst $(n=6)$, whenever $q>716.$

If $ \omega= 14$ then, proceeding in the same way as above, we see that \eqref{****} is satisfied with $\omega(l)=5$ and $r=9,$  for all $q$ and $n$ with $\omega=14$ or for $q > 460$ at worst when $n=6$.

  Next we assume $4 \leq\omega\leq 13$ then repeating the above process with $\omega(l)=4$  we get $\delta > 0.11815$ and $R <112040$. Hence \eqref{****} holds whenever $q^n>1.40643\times 10^{15}$. Hence  we conclude that $(q,n) \in \mathfrak{P}$ whenever   $q>(1.40643\times 10^{15})^{(1/n)}$; so that, whenever $q>334,$ for $n=6$;  $q>145,$ for $n=7$;  $q>78,$ for $n=8$; $q>48,$ for $n=9$; $q>32,$ for $n=10$; $q\geq 25,$ for $n=11$; $q>18,$ for $n=12$; $q>15,$ for $n=13$; $q>12,$ for $n=14$;  $q>10$ for $n = 15; $ and  $q\geq 9$ for $n \geq 16. $ Note that  if $\omega \leq 3$ even then the pairs discussed above satisfy \eqref{****} with $l=q^n-1$.  Factorizing $q^n-1$ for the remaining values of $q$ and $n$, we see that \eqref{****} is satisfied by these pairs $(q,n)$ for appropriate choices of $l$ except the pairs $(q,6)$ with $q \leq 25,$ $q=29,31,61$; $(3,7),(7,7),(3,8),(5,8),(3,9),(3,12)$. Some illustrative cases are given in Table \ref{t1}. Hence the result follows.

\end{proof}

We turn to the case in which $n=5$.  Here $q-1$ and $\frac{q^5-1}{q-1}$ are coprime unless $q \equiv 1 \mod \ 5$  when their highest common factor is 5.  Write $q_1$ for the factor of $q^5-1$ all of whose prime divisors divide $q-1$ and $q_2= (q^5-1)/q_1$.  Then all primes dividing $q_2$ are in the set $S$, defined as the set of primes congruent to 1 modulo 10.
  \begin{thm}\label{n5}
  Let $q=p^k$ for positive integer $k$ and odd prime $p$. Then $(q,5)\in\mathfrak{P}$ for $q \geq 17$
  except (possibly) for $q=19,25,31,37,43,49,61,71.$
  \end{thm}

  \begin{proof}
  By the same argument as in Theorem \ref{n6}, we see that \eqref{****} is satisfied for $\omega \geq 149$.

We proceed to the sieving argument.  In this case observe that  \eqref{****} is equivalent to $q^5 > R^{10/3}$.

Perform two steps of the sieving argument as in Theorem \ref{n6}  without regard to the special nature of the primes in $q_2$.  Specifically,  first assume  $\omega(l) =18 \leq \omega \leq 148$ and then $\omega(l) =8 \leq \omega \leq 31$. Consequently,  \eqref{****} is satisfied if $q^5> 3.39318\times10^{25},$ i.e., if $q> 127679.$  But if $\omega \geq 20$ then $q^5 >5.5794\times10^{26}$. Hence $(q,5) \in \mathfrak{P}$ for all $q$ with $\omega \geq 20,$ or $q>12679$.

  Hence it can be assumed that $\omega \leq 19,$ and $q \leq 126769$.  But $q \leq126769$ implies that $\omega(q_1) \leq 6.$
Moreover,  since all primes   dividing $q_2$ are in $S$, it follows that if $\omega(q_2) \geq 11$, then $q_2 >8.8245 \times 10^{20}$and so $q>172354$, whence $ (q,5) \in \mathfrak{P}$.

Hence we can assume $\omega \leq 16$ with $\omega(q_1) \leq 6$ and  $\omega(q_2) \leq 10$. Take $\omega(l)=4$. and $r \leq 12$. To obtain a minimum theoretical value for $\delta$,  regard $l$ as involving the first four primes $2,3,5,7$ and $\{p_1,\ldots, p_{12}\}$ as comprising the first 10 primes in $S$, namely $11, 31, \ldots, 191$, together with 13 and 17, the next two primes not in $S$.  This yields $\delta >0.30260 $ and $R<59910$, whence $(q,5) \in \mathfrak{P}$ whenever $q^5 > 8.4139 \times 10^{15}$, i.e., $q>1532$.  Further, if $\omega(q_2) \geq 8$ then $q_2> 1.2097 \times 10^{14}$ so that $q>3315$.  Thus, we can suppose that $\omega(q_2) \leq 7$ and $q<1532$ and therefore $\omega(q_1) \leq 3$ and $\omega \leq 10$.  Repeat the above step with $\omega(l) =3$   when the minimal value of $\delta$ is obtained when notionally $l$ is divisible by $2$ and $3$ and $p_1, \ldots, p_8$ comprise the prime 5 and the first 7 primes in $S$.  The outcome is that $\delta >0.20886$ and $R<3544$.
Hence Therefore $(q,5) \in \mathfrak{P}$ if $q^5 >6.7820 \times 10^{11}$ or $q>233$.  Now, as before, if $\omega(q_2) \geq 6$ this is bound to be the case.
We can therefore assume that  $\omega(q_2) \leq 5$ and $q<233$ so that certainly $\omega(q_1) \leq 3$ and $\omega \leq 8$.
One more cycle of the sieving argument with  $\omega(l)=2$  means we can assume that
 $q <173$.

 To complete the proof for odd prime powers $q$   we factorized $q^5-1$ and checked to see when \eqref{****} was satisfied for an  appropriate choice of $l$ (see Table \ref{t1}).   This was successful except for $3 \leq q\leq 13$ and $q=19,25,31,37,43,49,61,71$.
  \end{proof}

\section{Even prime powers $q$ and conclusions}

A Mersenne prime is a prime of the form $2^n-1$ for some positive integer $n.$
\begin{lemma}\label{kk}
If $2^n-1 \geq 7$  is a Mersenne prime then $(2,n) \in \mathfrak{P}.$
\end{lemma}
\begin{proof}	
If $2^n-1$ is a Mersenne prime, i.e., if $n=3,~5,~7,~13,~17,~19$ etc. then every  $\alpha\in\mathbb{F}_{2^n}^*$ other than $1$ is a primitive element of $\mathbb{F}_{2^n}$.
Also, if $\alpha \in \mathbb{F}_{2^n}^*$ then degree of its minimal polynomial over $\mathbb{F}_{2}$ is $n\geq3$. Hence $\alpha+\alpha^{-1} \neq 0,1.$ Thus $\alpha+\alpha^{-1}$ is also primitive. Moreover, the trace map $Tr_{\mathbb{F}_{2^n}|\mathbb{F}_{2}}$ is onto and inverse image of every element in $\mathbb{F}_2$ contains $2^{n-1}\geq4$ elements in $\mathbb{F}_{2^n}$ and at least three of them are primitive. Hence the result follows.
\end{proof}

\begin{thm}\label{result1}
Let $q=2^k$ for some positive integer $k$, and $n \geq 5$ be an integer. Then for every $ a \in \mathbb{F}_q$ there exists a primitive pair $(\alpha,\alpha+\alpha^{-1})$ in $\mathbb{F}_{q^n}$ such that $Tr_{\mathbb{F}_{q^n}|\mathbb{F}_{q}}(\alpha)=a $  if  $(q,n)$ is not one of the pairs $(2,12),(2,10),(2,9),(2,8),\\(2,6),(4,8),(4,7),(4,6), (4,5),(8,8),(8,6),(8,5),
(16,6), (16,5).$
\end{thm}
\begin{proof}
For even prime powers $q$  (in comparison with odd prime powers) arguments to verify the criteria of Theorem \ref{aa} are simplified, firstly, by the fact that now we have $C_q=2$, and, secondly, because $q^n-1$ is odd, so that $2$ is not a prime factor.    We  assume (for convenience just now) that $q\geq 8$, and   give only a brief outline based on Theorems \ref{n6} and \ref{n5}. (For $q=2, 4$, see below.)

Suppose $n\geq 6$ and $q \geq 8$   Assume first that $\omega \geq 149$.  By Lemma \ref{1/8} to satisfy inequality (\ref{****}), it suffices that $q^{n/2-2n/8-1}> 2$, i.e., $q^{n/4-1} >2$ which trivially holds.

Hence we can assume $\omega \leq 148$.  Perform a series of parallel steps as in the proof of Theorem \ref{n6}.  The first deals with $28 \leq \omega \leq 130$ with a choice of $\omega(l)=18$, the second with $17 \leq \omega \leq 27$ with $\omega(l) =7$, etc.

Eventually we reach the stage in which $3 \leq \omega \leq 12$ and $q< (6.5413\times 10^{15} )^{1/n}$.  With the choice of $\omega(l)=3$  and $r \leq 9$,  we conclude that  \eqref{****} is satisfied for all $q > 136$,  whenever $n=6$;
    for all $q>68,$  whenever $n=7$;  $q>40,$ whenever $n=8$; $q>27,$ whenever $n=9$; $q>20,$ whenever $n=10$; $q>15,$ whenever $n=11$; $q>12,$ whenever $n=12$; $q>10,$ whenever $n=13$; $q>8,$ whenever $n=14$;  $q\geq8$ whenever $n \geq  15. $  Note that  if $\omega \leq 3$ even then the pairs discussed above satisfy \eqref{****} with $l=q^n-1$.  Factorizing $q^n-1$ for the remaining values of $q$ and $n$, we see that \eqref{****} is satisfied by these pairs $(q,n)$ $(q\geq 8)$ for appropriate choices of $l$ except the pairs $(8,6),(8,8),(16.6)$. For delicate cases we refer to Table \ref{t1}.

    Moreover, for $q=2,4,$ and $n\leq 25,$ $\omega(q^n-1)$ is calculated and checked to see whether $q^{n/2-1} > 2\cdot 2^{2\omega}$ is satisfied, which is true  for $n \geq 19$ and $n=13,16,17$ when $q=4$; and for $n \geq 21$ except $n=24$ when $q=2.$    The pairs $(4,n),$ for $n=18,15,14,12,11,10,9$ satisfy sieving inequality \eqref{****} with appropriate choices of $l$. Hence $(4,n)\in \mathfrak{P}$ for every $n \geq 9.$ For $n=19,17,13,7,5$,  $2^n-1$ is a Mersenne prime, hence $(2,n)\in\mathfrak{P}$ by Lemma \ref{kk} for these values of $n$. Also $(2,n)\in \mathfrak{P},$ for $k=24,20,18,16,15,14,11$ as these satisfy sieving inequality \eqref{****} in Theorem  \ref{aa} by choosing some suitable $l$ except the pairs (2,12),(2,10),(2,9),(2,8),(2,6),(4,8),(4,7),(4,6).

Finally, for $n=5$ follow the argument of Theorem \ref{n5} (with $C_q=2$),  taking special account of the fact that primes (other than 5) dividing $\frac{q^5-1}{q-1}$ lie in the set $S$.    This yields $(q,5) \in \mathfrak{P}$ for $q \geq 256$. To complete the result it remains to verify the result for $ q\leq 128$. For these values of $q$ we factorize $q^5-1$ and  see that \eqref{***} is satisfied for $q=32,64,128$ (see Table \ref{t1}). For $q=2$, $2^5-1$ is a Mersenne prime. Hence $(q,5)\in \mathfrak{P}$ except for $q=4,8,16$.  \end{proof}
\begin{longtable}{|c|c|c|c|c|}
	\hline
	 $(q,n)$ & primes in $q^n-1$ & $ \omega(l)$ & $\delta$ & $R^{(2/(n-2))}$   \\
	\hline
\endhead
(2,28)& 3,5,29,43,113,127  & 2 & 0.8510 &1.5353   \\

	\hline
(23,5)& 2,11,292561  & 1 & 0.8181 &16.7325   \\
\hline
(27,5)& 2,11,13,4561  & 1 & 0.6638 &23.5644  \\
\hline
(47,5)&2,11,23,31,14621 & 1 & 0.6665 & 28.2351   \\
	\hline
(53,5)&2,11,13,131,5581 & 1 & 0.6487 & 28.6673   \\
	\hline
(59,5)&2,11,29,41,151,181 & 1 & 0.676 &32.3226   \\
	\hline
(67,5)&2,3,11,761,26881 & 2 & 0.8154 & 53.4103   \\
	\hline
%(81,5)&2,5,11,61,1181 & 2 & 0.7837 & 54.491   \\
%	\hline
%
%(89,5)&2,11,131,691,701 & 2 & 0.9789 &48.8259   \\
%	\hline
%(97,5)&2,3,11,31,262321 & 2 & 0.7536 &55.59  \\
	%\hline
%(101,5)&2,5,31,491,1381 & 2 & 0.8932 &51.0452   \\
%	\hline
%(103,5)&2,3,11,17,10332211 & 2 & 0.7005 &57.7282   \\
%	\hline
%(109,5)&2,3,31,191,24061 & 2 & 0.9249 &50.1823  \\
%	\hline
%(113,5)&2,7,11,251,59581 & 2 &0.8101 &53.5858  \\
%	\hline
%(131,5)&2,5,13,61,973001 & 2 & 0.8133 &53.4796  \\
%	\hline
%(133,5)&2,3,11,41,1321,5821 & 2 & 0.7675 &65.8028   \\
%	\hline
%(137,5)&2,11,17,101,319411 & 2 & 0.8247 &53.1071   \\
%	\hline
%(139,5)&2,3,23,41,9170881 & 2 & 0.8642 &51.8831   \\
%	\hline
%(149,5)&2,37,251,691,2861 & 2 &0.9884 &48.6008   \\
%	\hline
%(157,5)&2,3,11,13,31,1793161 & 2 & 0.5998 &75.5125  \\
%	\hline
%(163,5)&2,3,11,31,1301,1601 & 2 & 0.7508 &66.6007  \\
%	\hline
(169,5)&2,3,7,11,2411,30941 & 3 & 0.8172 &134.4364   \\
	\hline
(27,6)& 2,7,13,19,37,757  & 2&0.6841 & 24.2314  \\

	\hline
(37,6)& 2,3,7,19,31,43,67  & 2 & 0.4681 & 31.9199 \\
\hline
(41,6)&2,3,5,7,547,1723 & 2 & 0.3094 &34.3799   \\
	\hline
(43,6)&2,3,7,11,13,139,631 & 2 & 0.361 &35.9539   \\
	\hline
(47,6)&2,3,7,23,37,61,103 & 2& 0.5210 &30.4167   \\
	\hline
(49,6)&2,3,5,13,19,43,181& 3 & 0.6833&48.4864  \\
	\hline
(53,6)&2,3,7,13,409,919 & 3 & 0.8390 &39.0925   \\
	\hline
(59,6)&2,3,5,7,29,163,3541 & 3& 0.6324 &50.0923   \\
	\hline
(67,6)&2,3,7,11,17,31,4423 & 3 & 0.6355 &49.9888   \\
	\hline
(71,6)&2,3,5,7,1657,5113 & 3 & 0.7126 &41.6075  \\
	\hline
(73,6)&2,3,7,37,751,1801 & 3 & 0.9421 &37.4567   \\
	\hline
(79,6)&2,3,5,7,13,43,6163 & 3 & 0.5136 &54.7798  \\
	\hline
(11,7)&2,5,43,45319 & 2 &0.9534 &9.0596   \\
	\hline
(7,8)&2,3,5,1201 & 2 & 0.5983 &6.9568   \\
\hline
(9,8)& 2,5,17,41,193  & 2 & 0.8232 & 7.2908 \\
\hline
(11,8)&2,3,5,61,7321 & 2 & 0.5669&8.0382   \\
	\hline
(13,8)&2,3,5,7,17,14281 & 2 & 0.1964 &12.1797   \\
	\hline

%(19,8)& 2,3,5,17,181,3833  & 2 & 0.4707 & 9.3211  \\
%\hline
%(23,8)&2,3,5,11,53,139921 & 2 &0.3804 &9.9305   \\
%	
%\hline
%(27,8)& 2,5,7,13,41,73,6481  & 2 & 0.4839 & 9.9624  \\
%
%	\hline
%(31,8)& 2,3,5,13,37,409,1129  & 2 & 0.3854 &10.6763  \\
%
%	\hline
%(43,8)&2,3,5,7,11,17,37,193,521 & 3 & 0.3465 &18.6427   \\
%	\hline
%(47,8)&2,3,5,13,17,23,97,25153 & 3 & 0.6208&14.6861   \\
%	\hline
(5,9)&2,19,31,829 & 1 & 0.8278 &3.6897   \\
\hline
(7,9)&2,3,19,37,1063 & 2 & 0.8388 &5.4673  \\
\hline
(9,9)& 2,7,13,19,37,757  & 2 & 0.6814 & 6.1812  \\
\hline
(3,10)&2,11,61 & 1 & 0.7853 &2.8909  \\
	\hline
(5,10)&2,3,11,71,521 & 2 & 0.7861 &4.4758   \\
\hline
(7,10)&2,3,11,191,2801 & 2 & 0.8069 &4.4537   \\
\hline

(3,11)&2,23,3851 & 1  & 0.9125 &2.5151   \\
	\hline
%(3,12)&2,5,7,13,73 & 1 & 3 & 0.1330 &8.9426&1.1927   \\
%	\hline
(5,12)&2,3,7,13,31,601 & 2 & 0.4925 &3.7864   \\
\hline
%(7,12)&2,3,5,13,19,43,181 & 2& 0.2833 &4.3849   \\
%\hline

(11,12)&2,3,5,7,13,19,37,61,1117 & 3 & 0.3665 &5.7244  \\
	\hline

%(3,16)&2,5,17,41,193 & 1& 0.4232 &2.1644   \\
%
%\hline
%(3,18)&2,7,13,19,37,757 & 1 &0.3984 &1.9786   \\
%	\hline

(32,6)& 3,7,11,31,151,331 & 2 &  0.7343&19.2108 \\
	\hline
(64,6)& 3,5,7,13,19,37,73,109 & 2 &  0.3553 &32.4764 \\
	\hline

 (16,7)& 3,5,29,43,113,127&2 & 0.851&10.1377\\
 \hline
 (32,8)& 3,5,11,17,31,41,61681&2 & 0.5872&8.2154 \\
 \hline
 %(2,9)& 7,73& 0 & 0.6868& \\
% \hline
 (4,9)& 3,7,19,73&1 &0.5816&3.5558 \\
 \hline
 (16,9)& 3,5,7,13,19,37,73,109&2 & 0.3553&7.3073 \\
 \hline
 (4,10)& 3,5,11,31,41&2 & 0.7048&4.1303 \\
 \hline
 (8,10)& 3,7,11,31,151,331&2 &0.7343&4.3831 \\
 \hline
 (2,11)& 23,89&0 & 0.8905&1.6947 \\
 \hline
 (4,11)& 3,23,89,683&1 & 0.8876&2.4937 \\
 \hline
 (4,12)& 3,5,7,13,17,241&2 & 0.4344&3.5698 \\
 \hline
% (8,12)& 3,5,7,13,19,37,73,109&2  & 0.3553&4.0238 \\
% \hline
 (2,14)& 3,43,127&1 & 0.9377&1.8614 \\
 \hline
 %(4,14)& 3,5,29,43,113,127&2 & 0.8510&2.6251 \\
% \hline
 (2,15)& 7,31,151& 0 & 0.6365&1.582 \\
 \hline
 (4,15)& 3,7,11,31,151,331&2 & 0.7343&2.4828 \\
 \hline
(2,18)& 3,5,17,257&1 & 0.4745&1.9316 \\
 \hline
 %(4,18)& 3,5,7,13,19,37,73,109&2 &  0.3553&2.3873 \\
% \hline

%(2,20)&3,5,11,31,41 & 1 & 0.3048&1.8014   \\
%	\hline
	%(2,24)& 3,5,7,13,17,241  & 2 & 0.4344&1.7832 \\
%\hline

(64,5)&3,7,11,31,151,331 & 1 & 0.4486&31.4651   \\
	\hline

\caption{  Pairs $(q,n)$ satisfying \eqref{****}, i.e., $q> R^{(2/(n-2))}$}
\label{t1}
\end{longtable}

Combining the results of  Theorems \ref{r}, \ref{n6},  \ref{n5} and   \ref{result1}, we obtain our main result Theorem \ref{mr}.

For the exceptions listed in Theorem \ref{mr}, we have computationally verified the result using GAP 4r8\cite{gap}. %The calculations have been given in Appendix 'A'. The results of the program are obtained in the form $ (a,i)$,  where $ i$ is such that $x^i$ and $ x^i+x^{-1}$ ($x$ being a primitive element randomly chosen by GAP) are both primitive and $Tr(x^i)=a \in \mathbb{F}_q$.
%During the course of the computation it has been observed that there can be at most one pair, $(q,n)$ of exceptional values.
 Accordingly, we have established  Corollary \ref{cor1}.\\

Following Corollary \ref{cor1},  we have done some further computer verification for pairs $(q,n), n=3,4, q^n<2^{26} =6.7108\ldots \times10^7$. The longest time to verify a pair $(q,n)$  was about 22 minutes (for the pair $(401,3))$. Accordingly we end with a conjecture that  will be the focus of a subsequent study.

\begin{conjecture}\label{conj}
Let $q=p^k$ for some positive integer $k ,$ and  prime $p$. Also suppose that $n\geq3$ is a positive integer. Then with the exception  of the pairs $(q,n)=(3,3),(4,3),(5,3)$, for every $a\in \mathbb{F}_q,$ $\mathbb{F}_{q^n}$ contains a primitive element $\alpha$ such that $\alpha+\alpha^{-1}$ is also primitive and $Tr_{\mathbb{F}_{q^n}|\mathbb{F}_q}(\alpha)=a$.
(The excluded  pairs $(3,3),(4,3),(5,3)$ are true exceptions.)
\end{conjecture}

\textbf{Acknowledgment}: This work has been supported by CSIR, New Delhi, Govt. of India, Under
	Grant No. F.No. 09/086(1145)/2012-EMR-1.
%\section*{References}
%\bibliographystyle{plain}
%\bibliography{refi1}
%\include{refi1}

%\section*{References}

\end{document}